\documentclass[10pt]{article}%
\usepackage{amsfonts}
\usepackage{amsmath,amssymb}
\usepackage{amsmath}
\usepackage{amssymb}
\usepackage{graphicx}%
\providecommand{\U}[1]{\protect\rule{.1in}{.1in}}
\newtheorem{theorem}{Theorem}[section]

\newtheorem{definition}[theorem]{Definition}
\newtheorem{example}[theorem]{Example}

\newtheorem{remark}[theorem]{Remark}
\newenvironment{proof}[1][Proof]{\noindent \textbf{#1.} }{\  $\Box$}
\numberwithin{equation}{section}
\begin{document}

\title{A note on pricing of contingent claims under $G$-expectation}
\author{Mingshang Hu \thanks{School of Mathematics, Shandong University,
humingshang@sdu.edu.cn. Research supported by the National Natural Science
Foundation of China (11201262)}
\and Shaolin Ji\thanks{Qilu Institute of Finance, Shandong University,
jsl@sdu.edu.cn }}
\maketitle
\date{}

\textbf{Abstract}. In this paper, we study the pricing of contingent claims
under $G$-expectation. In order to accomodate volatility uncertainty, the
price of the risky security is supposed to governed by a general linear
stochastic differential equation (SDE) driven by $G$-Brownian motion.
Utilizing the recently developed results of Backward SDE driven by
$G$-Brownian motion, we obtain the superhedging and suberhedging prices of a
given contingent claim. Explicit results in the Markovian case are also derived.


\textbf{MSC-classification}: 60H30, 91G20

\section{Introduction}

It is well known that the Black-Scholes formula depends on the underlying
volatility. Since it is difficult to forecast the prospective volatility
process in practice, it is natural to permit volatility uncertainty in
contingent claim pricing models (see \cite{Avel1995}).

Motivated by measuring risk and other financial problems of volatility
uncertainty, Peng \cite{P07a} introduced the notion of sublinear expectation
space, which is a generalization of probability space. As a typical case, Peng
studied a fully nonlinear expectation, called $G$-expectation $\mathbb{\hat
{E}}\mathcal{[\cdot]}$ (see \cite{P10} and the references therein), and the
corresponding time-conditional expectation $\mathbb{\hat{E}}_{t}%
\mathcal{[\cdot]}$ on a space of random variables completed under the norm
$\mathbb{\hat{E}}[|\cdot|^{p}]^{1/p}$. Under this $G$-expectation framework
($G$-framework for short) a new type of Brownian motion called $G$-Brownian
motion was constructed. The stochastic calculus with respect to the
$G$-Brownian motion has been established. For a recent account and development
of $G$-expectation theory and its applications we refer the reader to
\cite{EJ-1, EJ-2,Nu, Peng2004, Peng2005, PengICM2010,PSZ2012, STZ, Song11,
Song12}.

There are other recent advances and their applications in stochastic calculus
which consists of mutually singular probability measures. For instance, Denis
and Martini \cite{DenisMartini2006} developed quasi-sure stochastic analysis
and Soner et al. \cite{STZ11} have obtained a deep result of existence and
uniqueness theorem of 2BSDE. Various stochastic control (game) problems and
the applications in finance are studied in \cite{MPP, MPZ, NN, Nutz, NZ}.

In this paper, we suppose that there are a riskless asset a risky security in
a financial market. Different from the existing literatures (see \cite{Bei,
EJ-2, vorbrink, XSZ}), the price $S_{t}$ to the risky security is governed by%
\[
dS_{t}=\eta_{t}S_{t}dt+\mu_{t}S_{t}d\langle B\rangle_{t}+\sigma_{t}S_{t}%
dB_{t}\text{,}%
\]
where $B$ is a $G$-Brownian motion. For a given contingent claim $\xi\in
L_{G}^{2}(\Omega)$ with maturity time $T$, we obtain its superhedging and
suberhedging prices. Explicit results in the Markovian case are also derived.
Our study bases on the recently developed BSDE driven by $G$-Brownian motion
in \cite{HJPS} and \cite{HJPS-1}:%
\begin{align*}
Y_{t}  &  =\xi+\int_{t}^{T}f(s,Y_{s},Z_{s})ds+\int_{t}^{T}g(s,Y_{s}%
,Z_{s})d\langle B\rangle_{s}\\
&  -\int_{t}^{T}Z_{s}dB_{s}-(K_{T}-K_{t}).
\end{align*}
We mainly utilize the existence and uniqueness theorem in \cite{HJPS} and some
important properties such as comparison theorem, Feynman-Kac formula and
Girsanov transformation in \cite{HJPS-1}.

The paper is organized as follows. In section 2, we formulate our contingent
claim pricing problem. The main results are given in section 3. In the
Appendix, we  present some fundamental results on $G$-expectation theory and
give proofs of the comparison theorem of SDE driven by $G$-Brownian motion and
the Girsanov transformation in our context.

\section{Statement of the problem}

There are a riskless asset with return $r_{t}$ and a risky security in a
financial market. The price $S_{t}$ to the risky securities is given by%
\begin{equation}
dS_{t}=\eta_{t}S_{t}dt+\mu_{t}S_{t}d\langle B\rangle_{t}+\sigma_{t}S_{t}%
dB_{t}\text{, }t\leq T\text{,}\label{Rt}%
\end{equation}
where $(\eta_{t})$, $(\mu_{t})$, $\left(  \sigma_{t}\right)  $ and $\left(
\sigma_{t}^{-1}\right)  $ are all bounded processes in $M_{G}^{2}(0,T)$. The
readers may refer to the Appendix to find the basic definitions and
fundamental results in the $G$-framework.

We denote the wealth process by $(Y_{t})$ and the amount of money invested in
the security by $(\psi_{t})$ at time $t$. Then the wealth process follows%
\begin{equation}
dY_{t}=r_{t}Y_{t}dt+\psi_{t}[(\eta_{t}-r_{t})dt+\mu_{t}d\langle B\rangle
_{t}]+\psi_{t}\sigma_{t}dB_{t}\text{.}\; \label{wealth-1}%
\end{equation}
Set $Z_{t}=\psi_{t}\sigma_{t}$, $b_{t}=\sigma_{t}^{-1}(\eta_{t}-r_{t})$ and
$d_{t}=\sigma_{t}^{-1}\mu_{t}$. Then (\ref{wealth-1}) becomes%
\begin{equation}
dY_{t}=r_{t}Y_{t}dt+b_{t}Z_{t}dt+d_{t}Z_{t}d\langle B\rangle_{t}+Z_{t}%
dB_{t}\text{,}\; \label{wealth}%
\end{equation}
here $Z$ is called the portfolio. In this note, we suppose that every $Z\in
M_{G}^{2}(0,T)$ is an admissible portfolio.

At the initial time $\tau\in\lbrack0,T]$, consider an investor with initial
wealth $\eta\in L_{G}^{2}(\Omega_{\tau})$ and denote by $Y^{\eta,Z,\tau}$ the
unique solution of the following SDE:%
\begin{equation}
\left\{
\begin{array}
[c]{l}%
dY_{t}^{\eta,Z,\tau}=r_{t}Y_{t}^{\eta,Z,\tau}dt+b_{t}Z_{t}dt+d_{t}%
Z_{t}d\langle B\rangle_{t}+Z_{t}dB_{t}\text{, }t\in\lbrack\tau,T]\text{,}\\
Y_{\tau}^{\eta,Z,\tau}=\eta\text{,}%
\end{array}
\right.  \text{ } \label{wealth-2}%
\end{equation}
where $Z\in M_{G}^{2}(\tau,T)$ is a given portfolio. For a contingent claim
$\xi\in L_{G}^{\beta}(\Omega_{T})$ with $\beta>2$, we define the superhedging
set
\[
\mathcal{U}_{\tau}=\{ \eta\in L_{G}^{2}(\Omega_{\tau}):\exists Z\in M_{G}%
^{2}(\tau,T)\text{ such that }Y_{T}^{\eta,Z,\tau}\geq\xi\text{, q.s.}\}
\]
and the \emph{superhedging price} $\overline{S}_{\tau}=ess\inf\{ \eta:\eta
\in\mathcal{U}_{\tau}\}$. Similarly define the subhedging set%
\[
\mathcal{L}_{\tau}=\{ \eta\in L_{G}^{2}(\Omega_{\tau}):\exists Z\in M_{G}%
^{2}(0,T)\text{ such that }Y_{T}^{-\eta,Z,\tau}\geq-\xi\text{, q.s.}\}
\]
and the \emph{subhedging price} $\underline{S}_{\tau}=ess\sup\{ \eta:\eta
\in\mathcal{L}_{\tau}\}$.

\begin{remark}
For $\tau=0$, $\mathcal{U}_{0}\subset\mathbb{R}$, thus $\overline{S}_{0}%
=\inf\{y\in\mathbb{R}:y\in\mathcal{U}_{0}\}$ is well defined. For $\tau>0$,
$\overline{S}_{\tau}=ess\inf\{ \eta:\eta\in\mathcal{U}_{\tau}\}$ is defined in
the following sense:

\begin{description}
\item[(1)] $\overline{S}_{\tau}\in L_{G}^{2}(\Omega_{\tau})$;

\item[(2)] For each $\eta\in\mathcal{U}_{\tau}$, we have $\eta\geq\overline
{S}_{\tau}$ q.s.;

\item[(3)] If $\zeta\in L_{G}^{2}(\Omega_{\tau})$ such that $\zeta\leq\eta$
q.s. for each $\eta\in\mathcal{U}_{\tau}$, then $\overline{S}_{\tau}\geq\zeta$ q.s..
\end{description}

In this note, we will show that $\overline{S}_{\tau}$ is well-posed which is
non-trivial due to the non-dominated probability measures in $\mathcal{P}$.
Similarly, $\underline{S}_{\tau}$ is well defined.
\end{remark}

\section{Main results}

\bigskip

\subsection{State price process}

We consider the following $G$-BSDE:%
\begin{equation}
Y_{t}=\xi-\int_{t}^{T}(r_{s}Y_{s}+b_{s}Z_{s})ds-\int_{t}^{T}d_{s}Z_{s}d\langle
B\rangle_{s}-\int_{t}^{T}Z_{s}dB_{s}-(K_{T}-K_{t})\text{, }t\leq T\text{.}
\label{main-BSDE}%
\end{equation}
In order to introduce the state price process which can be used to solve the
$G$-BSDE (\ref{main-BSDE}), we construct an auxiliary extended $\tilde{G}%
$-expectation space $(\tilde{\Omega}_{T},L_{\tilde{G}}^{2}(\tilde{\Omega}%
_{T}),\mathbb{\hat{E}}^{\tilde{G}})$ with $\tilde{\Omega}_{T}=C_{0}%
([0,T],\mathbb{R}^{2})$ and%

\[
\tilde{G}(A)=\frac{1}{2}\sup_{\underline{\sigma}^{2}\leq v\leq\bar{\sigma}%
^{2}}\mathrm{tr}\left[  A\left[
\begin{array}
[c]{cc}%
v & 1\\
1 & v^{-1}%
\end{array}
\right]  \right]  ,\ A\in\mathbb{S}_{2}.
\]
Let $\{(B_{t},\tilde{B}_{t})\}$ be the canonical process in the extended space
(see \cite{HJPS-1}). Note that $\langle B,\tilde{B}\rangle_{t}=t$.

By the state price process we mean the unique solution $\pi=\left(  \pi
_{t}\right)  $ to
\begin{equation}
d\pi_{t}/\pi_{t}=-r_{t}dt-b_{t}d\tilde{B}_{t}-d_{t}dB_{t}\text{,\ \ }\pi
_{0}=1\text{,} \label{Pi}%
\end{equation}
which admits a closed form (see \cite{HJPS-1}): for $0\leq t\leq T,$%
\[%
\begin{array}
[c]{rl}%
\pi_{t}= & \exp\{-\int\nolimits_{0}^{t}(r_{s}+b_{s}d_{s})ds\}\exp
\{-\int\nolimits_{0}^{t}b_{s}d\tilde{B}_{s}-\tfrac{1}{2}\int\nolimits_{0}%
^{t}b_{s}^{2}d\langle\tilde{B}\rangle_{s}\}\\
& \exp\{-\int\nolimits_{0}^{t}d_{s}dB_{s}-\tfrac{1}{2}\int\nolimits_{0}%
^{t}d_{s}^{2}d\langle B\rangle_{s}\}\text{.}%
\end{array}
\]
By applying It\^{o}'s formula to $\pi_{t}Y_{t}$, we obtain
\begin{equation}
Y_{t}=\mathbb{\hat{E}}_{t}^{\tilde{G}}[\frac{\pi_{T}}{\pi_{t}}\xi]\text{,
}t\leq T\text{.} \label{main-rel}%
\end{equation}

\subsection{Hedging prices}

\begin{theorem}
[Hedging prices]\label{thm-hedge} Let $\xi\in L_{G}^{2}(\Omega_{T})$ be a
contingent claim. Suppose that $(r_{t})$, $(\eta_{t})$, $(\mu_{t})$, $\left(
\sigma_{t}\right)  $ and $\left(  \sigma_{t}^{-1}\right)  $ are bounded
processes in $M_{G}^{2}(0,T)$. Then the superhedging and subhedging prices at
any time $\tau$ are given by
\[
\overline{S}_{\tau}=\mathbb{\hat{E}}_{\tau}^{\tilde{G}}[\frac{\pi_{T}}%
{\pi_{\tau}}\xi]
\]
and
\[
\underline{S}_{\tau}=-\mathbb{\hat{E}}_{\tau}^{\tilde{G}}[-\frac{\pi_{T}}%
{\pi_{\tau}}\xi]\text{.}%
\]

\end{theorem}

\textbf{Proof.} By the definition of subhedging price, it is easy to get
$\underline{S}_{\tau}$ from the superhedging price $\overline{S}_{\tau}$. Thus
we only need to prove the superhedging price.

\medskip

\noindent\textbf{Step 1:} We first show that for any $\eta\in\mathcal{U}%
_{\tau}$,
\[
\eta\geq\mathbb{\hat{E}}_{\tau}^{\tilde{G}}[\frac{\pi_{T}}{\pi_{\tau}}%
\xi]\text{, q.s..}%
\]

\noindent If $\eta\in\mathcal{U}_{\tau}$, then there exists a $Z\in M_{G}%
^{2}(\tau,T)$ such that $Y_{T}^{\eta,Z,\tau}\geq\xi$. Thus%
\begin{equation}
Y_{t}^{\eta,Z,\tau}=Y_{T}^{\eta,Z,\tau}-\int_{t}^{T}(r_{s}Y_{s}^{\eta,Z,\tau
}+b_{s}Z_{s})ds-\int_{t}^{T}d_{s}Z_{s}d\langle B\rangle_{s}-\int_{t}^{T}%
Z_{s}dB_{s}\text{, }t\in\lbrack\tau,T]\text{.} \label{3-1}%
\end{equation}
Let $(\bar{Y}_{t},\bar{Z}_{t},\bar{K}_{t})_{t\leq T}$ be the solution of the
$G$-BSDE (\ref{main-BSDE}) corresponding to the terminal value $Y_{T}%
^{\eta,Z,\tau}$. Then by (\ref{3-1}) we get $(\bar{Y}_{t},\bar{Z}_{t},\bar
{K}_{t})=(Y_{t}^{\eta,Z,\tau},Z_{t},0)$ for $t\in\lbrack\tau,T]$. Let
$(\tilde{Y}_{t},\tilde{Z}_{t},\tilde{K}_{t})_{t\leq T}$ be the solution of the
$G$-BSDE (\ref{main-BSDE}) corresponding to the terminal value $\xi$. Then by
(\ref{main-rel}) we have $\tilde{Y}_{t}=\mathbb{\hat{E}}_{t}^{\tilde{G}}%
[\frac{\pi_{T}}{\pi_{t}}\xi]$ for $t\leq T$. Note that $Y_{T}^{\eta,Z,\tau
}\geq\xi$, then by the comparison theorem of $G$-BSDEs (see \cite{HJPS-1}) we
obtain%
\[
\bar{Y}_{\tau}=Y_{\tau}^{\eta,Z,\tau}=\eta\geq\tilde{Y}_{\tau}=\mathbb{\hat
{E}}_{\tau}^{\tilde{G}}[\frac{\pi_{T}}{\pi_{\tau}}\xi]\text{, q.s..}%
\]
\noindent\textbf{Step 2: }We now prove that $\mathbb{\hat{E}}_{\tau}%
^{\tilde{G}}[\frac{\pi_{T}}{\pi_{\tau}}\xi]=\tilde{Y}_{\tau}\in$
$\mathcal{U}_{\tau}$.

For this purpose, we consider the following wealth process $(\hat{Y}%
_{t})_{t\in\lbrack\tau,T]}$ with the initial wealth $\tilde{Y}_{\tau}$ and
portfolio $\tilde{Z}$:%
\[
\hat{Y}_{t}=\tilde{Y}_{\tau}+\int_{\tau}^{t}(r_{s}\hat{Y}_{s}+b_{s}\tilde
{Z}_{s})ds+\int_{\tau}^{t}d_{s}\tilde{Z}_{s}d\langle B\rangle_{s}+\int_{\tau
}^{t}\tilde{Z}_{s}dB_{s}\text{, }t\in\lbrack\tau,T]\text{.}%
\]
On the other hand, $(\tilde{Y}_{t},\tilde{Z}_{t},\tilde{K}_{t})_{t\leq T}$ is
the solution of the $G$-BSDE (\ref{main-BSDE}) corresponding to the terminal
value $\xi$. Thus we get%
\[
\tilde{Y}_{t}=\tilde{Y}_{\tau}+\int_{\tau}^{t}(r_{s}\tilde{Y}_{s}+b_{s}%
\tilde{Z}_{s})ds+\int_{\tau}^{t}d_{s}\tilde{Z}_{s}d\langle B\rangle_{s}%
+\int_{\tau}^{t}\tilde{Z}_{s}dB_{s}+\tilde{K}_{t}-\tilde{K}_{\tau}\text{,
}t\in\lbrack\tau,T]\text{.}%
\]
Note that $\tilde{K}$ is a decreasing process, then by the comparison theorem
of SDE (see Appendix) we obtain $\hat{Y}_{T}\geq\tilde{Y}_{T}=\xi$ q.s., which
implies that $\tilde{Y}_{\tau}\in$ $\mathcal{U}_{\tau}$.

This completes the proof. $\blacksquare$

\begin{remark}
In the special case where $\xi$ can be perfectly hedged, that is, there exist
$y$ and $Z$ such that $Y_{T}^{y,Z,0}=\xi$, then
\[
\overline{S}_{0}=\underline{S}_{0}=\mathbb{\hat{E}}_{t}^{\tilde{G}}[\frac
{\pi_{T}}{\pi_{\tau}}\xi]=-\mathbb{\hat{E}}_{t}^{\tilde{G}}[-\frac{\pi_{T}%
}{\pi_{\tau}}\xi]\text{.}%
\]

\end{remark}

\begin{remark}
Vorbrink (2010) obtains a characterization of hedging prices under
$G$-expectation. However, in place of our assumption, he adopts the strong
assumption that $\eta_{t}=r_{t}$ and $\mu_{t}=0$, so $\pi_{t}=\exp
\{-\int\nolimits_{0}^{t}r_{s}ds\}$.
\end{remark}

\bigskip

\subsection{ Some special cases}

Suppose that $(r_{t})$, $\left(  \sigma_{t}\right)  $ and $\left(  \sigma
_{t}^{-1}\right)  $ are deterministic continuous functions on the time
interval $[0,T]$. $\xi=\Phi(S_{T})$ is a contingent claim, where
$\Phi:\mathbb{R}\rightarrow\mathbb{R}$ is a local Lipschitz function, i.e.,
there exist a constant $L>0$ and an positive integer $m$ such that%
\[
|\Phi(x)-\Phi(x^{\prime})|\leq L(1+|x|^{m}+|x^{\prime}|^{m})|x-x^{\prime}|.
\]
We consider the following $G$-BSDEs:%
\begin{equation}
Y_{t}=\Phi(S_{T})-\int_{t}^{T}(r_{s}Y_{s}+b_{s}Z_{s})ds-\int_{t}^{T}d_{s}%
Z_{s}d\langle B\rangle_{s}-\int_{t}^{T}Z_{s}dB_{s}-(K_{T}-K_{t}),
\label{sup-bsde}%
\end{equation}%
\begin{equation}
\bar{Y}_{t}=-\Phi(S_{T})-\int_{t}^{T}(r_{s}\bar{Y}_{s}+b_{s}\bar{Z}%
_{s})ds-\int_{t}^{T}d_{s}\bar{Z}_{s}d\langle B\rangle_{s}-\int_{t}^{T}\bar
{Z}_{s}dB_{s}-(\bar{K}_{T}-\bar{K}_{t}). \label{sub-bsde}%
\end{equation}
By (\ref{main-rel}) and Theorem \ref{thm-hedge}, we have $\overline{S}_{\tau
}=Y_{\tau}$ and $\underline{S}_{\tau}=-\bar{Y}_{\tau}$.

By applying It\^{o}'s formula to $\exp\{-\int\nolimits_{0}^{t}r_{s}ds\}Y_{t}$,
we obtain that $\tilde{Y}_{t}=\exp\{-\int\nolimits_{0}^{t}r_{s}ds\}Y_{t}$,
$\tilde{Z}_{t}=\exp\{-\int\nolimits_{0}^{t}r_{s}ds\}Z_{t}$ and $\tilde{K}%
_{t}=\int_{0}^{t}\exp\{-\int\nolimits_{0}^{u}r_{s}ds\}dK_{u}$ is the solution
of the following $G$-BSDE:%
\begin{equation}
\tilde{Y}_{t}=\exp\{-\int\nolimits_{0}^{T}r_{s}ds\} \Phi(S_{T})-\int_{t}%
^{T}b_{s}\tilde{Z}_{s}ds-\int_{t}^{T}d_{s}\tilde{Z}_{s}d\langle B\rangle
_{s}-\int_{t}^{T}\tilde{Z}_{s}dB_{s}-(\tilde{K}_{T}-\tilde{K}_{t}).
\label{3-2}%
\end{equation}
By the Girsanov transformation (see Appendix), we can define a consistent
sublinear expectation $(\mathbb{\tilde{E}}_{t}[\cdot])_{t\leq T}$ such that
$\tilde{B}_{t}=B_{t}+\int_{0}^{t}b_{s}ds+\int_{0}^{t}d_{s}d\langle
B\rangle_{s}$ is a $G$-Brownian motion and $\tilde{K}_{t}$ is a martingale
under $\mathbb{\tilde{E}}$. Thus equation (\ref{3-2}) becomes%
\begin{equation}
\tilde{Y}_{t}+(\tilde{K}_{T}-\tilde{K}_{t})=\exp\{-\int\nolimits_{0}^{T}%
r_{s}ds\} \Phi(S_{T})-\int_{t}^{T}\tilde{Z}_{s}d\tilde{B}_{s}. \label{3-3}%
\end{equation}
Taking $\mathbb{\tilde{E}}_{t}$ on both sides of equation (\ref{3-3}), we
obtain%
\begin{align*}
\tilde{Y}_{t}  &  =\mathbb{\tilde{E}}_{t}[\exp\{-\int\nolimits_{0}^{T}%
r_{s}ds\} \Phi(S_{T})]\\
&  =\exp\{-\int\nolimits_{0}^{T}r_{s}ds\} \mathbb{\tilde{E}}_{t}[\Phi
(S_{T})]\\
&  =\exp\{-\int\nolimits_{0}^{T}r_{s}ds\} \mathbb{\tilde{E}}_{t}[\Phi
(S_{t}\exp(\int_{t}^{T}r_{s}ds-\frac{1}{2}\int_{t}^{T}\sigma_{s}^{2}d\langle
B\rangle_{s}+\int_{t}^{T}\sigma_{s}d\tilde{B}_{s}))]\\
&  =\exp\{-\int\nolimits_{0}^{T}r_{s}ds\} \mathbb{\tilde{E}}_{t}[\Phi
(S_{t}\exp(\int_{t}^{T}r_{s}ds-\frac{1}{2}\int_{t}^{T}\sigma_{s}^{2}%
d\langle\tilde{B}\rangle_{s}+\int_{t}^{T}\sigma_{s}d\tilde{B}_{s}))]\\
&  =\exp\{-\int\nolimits_{0}^{T}r_{s}ds\} \mathbb{\tilde{E}}_{t}[\Phi
(x\exp(\int_{t}^{T}r_{s}ds-\frac{1}{2}\int_{t}^{T}\sigma_{s}^{2}d\langle
\tilde{B}\rangle_{s}+\int_{t}^{T}\sigma_{s}d\tilde{B}_{s}))]_{x=S_{t}}\\
&  =\exp\{-\int\nolimits_{0}^{T}r_{s}ds\} \mathbb{\tilde{E}}[\Phi(x\exp
(\int_{t}^{T}r_{s}ds-\frac{1}{2}\int_{t}^{T}\sigma_{s}^{2}d\langle\tilde
{B}\rangle_{s}+\int_{t}^{T}\sigma_{s}d\tilde{B}_{s}))]_{x=S_{t}}\\
&  =\exp\{-\int\nolimits_{0}^{T}r_{s}ds\} \mathbb{\hat{E}}[\Phi(x\exp(\int%
_{t}^{T}r_{s}ds-\frac{1}{2}\int_{t}^{T}\sigma_{s}^{2}d\langle B\rangle
_{s}+\int_{t}^{T}\sigma_{s}dB_{s}))]_{x=S_{t}}.
\end{align*}
Thus we can get the following theorem.

\begin{theorem}
Suppose that $(r_{t})$, $\left(  \sigma_{t}\right)  $ and $\left(  \sigma
_{t}^{-1}\right)  $ are deterministic continuous functions on the time
interval $[0,T]$. Let $\xi=\Phi(S_{T})$ be a contingent claim, where
$\Phi:\mathbb{R}\rightarrow\mathbb{R}$ is a local Lipschitz function. Then%
\begin{equation}
\overline{S}_{\tau}=\exp\{-\int\nolimits_{\tau}^{T}r_{s}ds\} \mathbb{\hat{E}%
}[\Phi(x\exp(\int_{\tau}^{T}r_{s}ds-\frac{1}{2}\int_{\tau}^{T}\sigma_{s}%
^{2}d\langle B\rangle_{s}+\int_{\tau}^{T}\sigma_{s}dB_{s}))]_{x=S_{t}}
\label{3-4}%
\end{equation}
and%
\begin{equation}
\underline{S}_{\tau}=-\exp\{-\int\nolimits_{\tau}^{T}r_{s}ds\} \mathbb{\hat
{E}}[-\Phi(x\exp(\int_{\tau}^{T}r_{s}ds-\frac{1}{2}\int_{\tau}^{T}\sigma
_{s}^{2}d\langle B\rangle_{s}+\int_{\tau}^{T}\sigma_{s}dB_{s}))]_{x=S_{t}}.
\label{3-5}%
\end{equation}

\end{theorem}

For each $(\tau,x)\in\lbrack0,T]\times\mathbb{R}$, we set%
\begin{equation}
u(\tau,x)=\exp\{-\int\nolimits_{\tau}^{T}r_{s}ds\} \mathbb{\hat{E}}[\Phi
(x\exp(\int_{\tau}^{T}r_{s}ds-\frac{1}{2}\int_{\tau}^{T}\sigma_{s}^{2}d\langle
B\rangle_{s}+\int_{\tau}^{T}\sigma_{s}dB_{s})). \label{3-6}%
\end{equation}

Then $\overline{S}_{\tau}=u(\tau,S_{t})$ and $u$ is the unique viscosity
solution of the following PDE (see Theorem 4.5 in \cite{HJPS-1}):%
\[
\left\{
\begin{array}
[c]{l}%
\partial_{t}u+G((\sigma_{t}x)^{2}\partial_{xx}^{2}u)+r_{t}x\partial_{x}%
u-r_{t}u=0,\\
u(T,x)=\Phi(x).
\end{array}
\right.
\]

\begin{example}
\label{example-hedging} We study the super and subhedging prices of a European
call option. Let the parameters in equation (\ref{Rt}) be constants, i.e.
\[
\eta_{t}:=\eta,\; \mu_{t}:=\mu\; \text{and }\sigma_{t}:=1.
\]
Then the price process $\left(  S_{t}\right)  $ becomes
\[
dS_{t}=\eta S_{t}dt+\mu S_{t}d\langle B\rangle_{t}+S_{t}dB_{t}\text{.}%
\]
Suppose further that $r_{t}\equiv r$ and $r$ is a constant. Thus $b_{t}%
=\eta-r,$ $d_{t}=\mu$ and the state price is%
\[
\pi_{t}=\exp\{-\mu(\eta-r)t\} \exp\{-rt-(\eta-r)\tilde{B}_{t}-\tfrac{1}%
{2}(\eta-r)^{2}\langle\tilde{B}\rangle_{t}\} \exp\{-\mu B_{t}-\tfrac{1}{2}%
\mu^{2}\langle B\rangle_{t}\} \text{.}%
\]

Consider a European call option on the risky security that matures at date $T$
and has exercise price $K$. The super and subhedging prices at $t$ can be
written in the form $\overline{c}(S_{t},t)$ and $\underline{c}(S_{t},t)$
respectively. At the maturity date,
\[
\overline{c}(S_{T},T)=\underline{c}(S_{T},T)=\max[0,S_{T}-K]\equiv\Phi
(S_{T}).
\]

By Theorem \ref{thm-hedge},
\[
\overline{c}(S_{t},t)=\mathbb{\hat{E}}_{t}^{\tilde{G}}[\frac{\pi_{T}}{\pi_{t}%
}\Phi(S_{T})]
\]
and
\[
\underline{c}(S_{t},t)=-\mathbb{\hat{E}}_{t}^{\tilde{G}}[-\frac{\pi_{T}}%
{\pi_{t}}\Phi(S_{T})]\text{.}%
\]
By the PDE approach, we obtain the following equations:
\[
\partial_{t}\overline{c}+\sup_{\underline{\sigma}^{2}\leq v\leq\overline
{\sigma}^{2}}\{ \tfrac{1}{2}vS^{2}\partial_{SS}^{2}\overline{c}\}+rS\partial
_{S}\overline{c}-r\overline{c}=0,\; \overline{c}(S,T)=\Phi(S)
\]
and%
\[
\partial_{t}\underline{c}-\sup_{\underline{\sigma}^{2}\leq v\leq
\overline{\sigma}^{2}}\{-\tfrac{1}{2}vS^{2}\partial_{SS}^{2}\underline{c}%
\}+rS\partial_{S}\underline{c}-r\underline{c}=0,\; \underline{c}%
(S,T)=\Phi(S).
\]
Because $\Phi(\cdot)$ is convex, so is $\overline{c}(\cdot,t)$. It follows
that the respective suprema in the above equations are achieved at
$\overline{\sigma}^{2}$ and \underline{$\sigma$}$^{2}$, and we obtain%
\[
\partial_{t}\overline{c}+\tfrac{1}{2}\overline{\sigma}^{2}S^{2}\partial
_{SS}^{2}\overline{c}+rS\partial_{S}\overline{c}-r\overline{c}=0,\;
\overline{c}(S,T)=\Phi(S)
\]
and%
\[
\partial_{t}\underline{c}+\tfrac{1}{2}\underline{\sigma}^{2}S^{2}\partial
_{SS}^{2}\underline{c}+rS\partial_{S}\underline{c}-r\underline{c}=0,\;
\underline{c}(S,T)=\Phi(S).
\]
Therefore,
\[
\overline{c}(S_{t},t)=E^{P^{\overline{\sigma}}}[\frac{\pi_{T}}{\pi_{t}}%
\Phi(S_{T}))\mid\mathcal{F}_{t}]
\]
and
\[
\underline{c}(S_{t},t)=E^{P^{\underline{\sigma}}}[\frac{\pi_{T}}{\pi_{t}}%
\Phi(S_{T})\mid\mathcal{F}_{t}]\text{.}%
\]
In other words, the super and subhedging prices are the Black-Scholes prices
with volatilities $\overline{\sigma}$ and $\underline{\sigma}$
respectively.{\LARGE \ }
\end{example}

\begin{remark}
In the above example, we find that the super and subhedging prices are
independent of $\eta$ and $\mu$.
\end{remark}

\bigskip

\section{Appendix}

We review some basic notions and results of $G$-expectation, the related
spaces of random variables and the backward stochastic differential equations
driven by a $G$-Browninan motion. The readers may refer to \cite{HJPS},
\cite{P07a}, \cite{P07b}, \cite{P08a}, \cite{P08b}, \cite{P10} for more details.

\begin{definition}
\label{def2.1} Let $\Omega$ be a given set and let $\mathcal{H}$ be a vector
lattice of real valued functions defined on $\Omega$, namely $c\in\mathcal{H}$
for each constant $c$ and $|X|\in\mathcal{H}$ if $X\in\mathcal{H}$.
$\mathcal{H}$ is considered as the space of random variables. A sublinear
expectation $\mathbb{\hat{E}}$ on $\mathcal{H}$ is a functional $\mathbb{\hat
{E}}:\mathcal{H}\rightarrow\mathbb{R}$ satisfying the following properties:
for all $X,Y\in\mathcal{H}$, we have

\begin{description}
\item[(a)] Monotonicity: If $X\geq Y$ then $\mathbb{\hat{E}}[X]\geq
\mathbb{\hat{E}}[Y]$;

\item[(b)] Constant preservation: $\mathbb{\hat{E}}[c]=c$;

\item[(c)] Sub-additivity: $\mathbb{\hat{E}}[X+Y]\leq\mathbb{\hat{E}%
}[X]+\mathbb{\hat{E}}[Y]$;

\item[(d)] Positive homogeneity: $\mathbb{\hat{E}}[\lambda X]=\lambda
\mathbb{\hat{E}}[X]$ for each $\lambda\geq0$.
\end{description}

$(\Omega,\mathcal{H},\mathbb{\hat{E}})$ is called a sublinear expectation space.
\end{definition}

\begin{definition}
\label{def2.2} Let $X_{1}$ and $X_{2}$ be two $n$-dimensional random vectors
defined respectively in sublinear expectation spaces $(\Omega_{1}%
,\mathcal{H}_{1},\mathbb{\hat{E}}_{1})$ and $(\Omega_{2},\mathcal{H}%
_{2},\mathbb{\hat{E}}_{2})$. They are called identically distributed, denoted
by $X_{1}\overset{d}{=}X_{2}$, if $\mathbb{\hat{E}}_{1}[\varphi(X_{1}%
)]=\mathbb{\hat{E}}_{2}[\varphi(X_{2})]$, for all$\ \varphi\in C_{l.Lip}%
(\mathbb{R}^{n})$, where $C_{l.Lip}(\mathbb{R}^{n})$ is the space of real
continuous functions defined on $\mathbb{R}^{n}$ such that
\[
|\varphi(x)-\varphi(y)|\leq C(1+|x|^{k}+|y|^{k})|x-y|\ \text{\ for
all}\ x,y\in\mathbb{R}^{n},
\]
where $k$ and $C$ depend only on $\varphi$.
\end{definition}

\begin{definition}
\label{def2.3} In a sublinear expectation space $(\Omega,\mathcal{H}%
,\mathbb{\hat{E}})$, a random vector $Y=(Y_{1},\cdot\cdot\cdot,Y_{n})$,
$Y_{i}\in\mathcal{H}$, is said to be independent of another random vector
$X=(X_{1},\cdot\cdot\cdot,X_{m})$, $X_{i}\in\mathcal{H}$ under $\mathbb{\hat
{E}}[\cdot]$, denoted by $Y\bot X$, if for every test function $\varphi\in
C_{l.Lip}(\mathbb{R}^{m}\times\mathbb{R}^{n})$ we have $\mathbb{\hat{E}%
}[\varphi(X,Y)]=\mathbb{\hat{E}}[\mathbb{\hat{E}}[\varphi(x,Y)]_{x=X}]$.
\end{definition}

\begin{definition}
\label{def2.4} ($G$-normal distribution) A $d$-dimensional random vector
$X=(X_{1},\cdot\cdot\cdot,X_{d})$ in a sublinear expectation space
$(\Omega,\mathcal{H},\mathbb{\hat{E}})$ is called $G$-normally distributed if
for each $a,b\geq0$ we have
\[
aX+b\bar{X}\overset{d}{=}\sqrt{a^{2}+b^{2}}X,
\]
where $\bar{X}$ is an independent copy of $X$, i.e., $\bar{X}\overset{d}{=}X$
and $\bar{X}\bot X$. Here the letter $G$ denotes the function
\[
G(A):=\frac{1}{2}\mathbb{\hat{E}}[\langle AX,X\rangle]:\mathbb{S}%
_{d}\rightarrow\mathbb{R},
\]
where $\mathbb{S}_{d}$ denotes the collection of $d\times d$ symmetric matrices.
\end{definition}

Peng \cite{P08b} showed that $X=(X_{1},\cdot\cdot\cdot,X_{d})$ is $G$-normally
distributed if and only if for each $\varphi\in C_{l.Lip}(\mathbb{R}^{d})$,
$u(t,x):=\mathbb{\hat{E}}[\varphi(x+\sqrt{t}X)]$, $(t,x)\in\lbrack
0,\infty)\times\mathbb{R}^{d}$, is the solution of the following $G$-heat
equation:%
\[
\partial_{t}u-G(D_{x}^{2}u)=0,\ u(0,x)=\varphi(x).
\]

The function $G(\cdot):\mathbb{S}_{d}\rightarrow\mathbb{R}$ is a monotonic,
sublinear mapping on $\mathbb{S}_{d}$ and $G(A)=\frac{1}{2}\mathbb{\hat{E}%
}[\langle AX,X\rangle]\leq\frac{1}{2}|A|\mathbb{\hat{E}}[|X|^{2}]=:\frac{1}%
{2}|A|\bar{\sigma}^{2}$ implies that there exists a bounded, convex and closed
subset $\Gamma\subset\mathbb{S}_{d}^{+}$ such that
\[
G(A)=\frac{1}{2}\sup_{\gamma\in\Gamma}\mathrm{tr}[\gamma A],
\]
where $\mathbb{S}_{d}^{+}$ denotes the collection of nonnegative elements in
$\mathbb{S}_{d}$.

In this paper, we only consider non-degenerate $G$-normal distribution, i.e.,
there exists some $\underline{\sigma}^{2}>0$ such that $G(A)-G(B)\geq
\underline{\sigma}^{2}\mathrm{tr}[A-B]$ for any $A\geq B$.

\begin{definition}
\label{def2.5} i) Let $\Omega_{T}=C_{0}([0,T];\mathbb{R}^{d})$, the space of
real valued continuous functions on $[0,T]$ with $\omega_{0}=0$, be endowed
with the supremum norm and let $B_{t}(\omega)=\omega_{t}$ be the canonical
process. Set
\[
\mathcal{H}_{T}^{0}:=\{ \varphi(B_{t_{1}},...,B_{t_{n}}):n\geq1,t_{1}%
,...,t_{n}\in\lbrack0,T],\varphi\in C_{l.Lip}(\mathbb{R}^{d\times n})\}.
\]
Let $G:\mathbb{S}_{d}\rightarrow\mathbb{R}$ be a given monotonic and sublinear
function. $G$-expectation is a sublinear expectation defined by
\[
\mathbb{\hat{E}}[X]=\mathbb{\tilde{E}}[\varphi(\sqrt{t_{1}-t_{0}}\xi_{1}%
,\cdot\cdot\cdot,\sqrt{t_{m}-t_{m-1}}\xi_{m})],
\]
for all $X=\varphi(B_{t_{1}}-B_{t_{0}},B_{t_{2}}-B_{t_{1}},\cdot\cdot
\cdot,B_{t_{m}}-B_{t_{m-1}})$, where $\xi_{1},\cdot\cdot\cdot,\xi_{n}$ are
identically distributed $d$-dimensional $G$-normally distributed random
vectors in a sublinear expectation space $(\tilde{\Omega},\tilde{\mathcal{H}%
},\mathbb{\tilde{E}})$ such that $\xi_{i+1}$ is independent of $(\xi_{1}%
,\cdot\cdot\cdot,\xi_{i})$ for every $i=1,\cdot\cdot\cdot,m-1$. The
corresponding canonical process $B_{t}=(B_{t}^{i})_{i=1}^{d}$ is called a
$G$-Brownian motion.

ii) Let us define the conditional $G$-expectation $\mathbb{\hat{E}}_{t}$ of
$\xi\in\mathcal{H}_{T}^{0}$ knowing $\mathcal{H}_{t}^{0}$, for $t\in
\lbrack0,T]$. Without loss of generality we can assume that $\xi$ has the
representation $\xi=\varphi(B_{t_{1}}-B_{t_{0}},B_{t_{2}}-B_{t_{1}},\cdot
\cdot\cdot,B_{t_{m}}-B_{t_{m-1}})$ with $t=t_{i}$, for some $1\leq i\leq m$,
and we put
\[
\mathbb{\hat{E}}_{t_{i}}[\varphi(B_{t_{1}}-B_{t_{0}},B_{t_{2}}-B_{t_{1}}%
,\cdot\cdot\cdot,B_{t_{m}}-B_{t_{m-1}})]
\]%
\[
=\tilde{\varphi}(B_{t_{1}}-B_{t_{0}},B_{t_{2}}-B_{t_{1}},\cdot\cdot
\cdot,B_{t_{i}}-B_{t_{i-1}}),
\]
where
\[
\tilde{\varphi}(x_{1},\cdot\cdot\cdot,x_{i})=\mathbb{\hat{E}}[\varphi
(x_{1},\cdot\cdot\cdot,x_{i},B_{t_{i+1}}-B_{t_{i}},\cdot\cdot\cdot,B_{t_{m}%
}-B_{t_{m-1}})].
\]

\end{definition}

Define $\Vert\xi\Vert_{p,G}=(\mathbb{\hat{E}}[|\xi|^{p}])^{1/p}$ for $\xi
\in\mathcal{H}_{T}^{0}$ and $p\geq1$. Then \textmd{for all}$\ t\in\lbrack
0,T]$, $\mathbb{\hat{E}}_{t}[\cdot]$ is a continuous mapping on $\mathcal{H}%
_{T}^{0}$ w.r.t. the norm $\Vert\cdot\Vert_{1,G}$. Therefore it can be
extended continuously to the completion $L_{G}^{1}(\Omega_{T})$ of
$\mathcal{H}_{T}^{0}$ under the norm $\Vert\cdot\Vert_{1,G}$.

Let $L_{ip}(\Omega_{T}):=\{ \varphi(B_{t_{1}},...,B_{t_{n}}):n\geq
1,t_{1},...,t_{n}\in\lbrack0,T],\varphi\in C_{b.Lip}(\mathbb{R}^{d\times
n})\},$ where $C_{b.Lip}(\mathbb{R}^{d\times n})$ denotes the set of bounded
Lipschitz functions on $\mathbb{R}^{d\times n}$. Denis et al. \cite{DHP11}
proved that the completions of $C_{b}(\Omega_{T})$ (the set of bounded
continuous function on $\Omega_{T}$), $\mathcal{H}_{T}^{0}$ and $L_{ip}%
(\Omega_{T})$ under $\Vert\cdot\Vert_{p,G}$ are the same and we denote them by
$L_{G}^{p}(\Omega_{T})$.

For each fixed $\mathbf{a}\in\mathbb{R}^{d}$, $B_{t}^{\mathbf{a}}%
=\langle\mathbf{a},B_{t}\rangle$ is a $1$-dimensional $G_{\mathbf{a}}%
$-Brownian motion, where $G_{\mathbf{a}}(\alpha)=\frac{1}{2}(\sigma
_{\mathbf{aa}^{T}}^{2}\alpha^{+}-\sigma_{-\mathbf{aa}^{T}}^{2}\alpha^{-})$,
$\sigma_{\mathbf{aa}^{T}}^{2}=2G(\mathbf{aa}^{T})$, $\sigma_{-\mathbf{aa}^{T}%
}^{2}=-2G(-\mathbf{aa}^{T})$. Let $\pi_{t}^{N}=\{t_{0}^{N},\cdots,t_{N}^{N}%
\}$, $N=1,2,\cdots$, be a sequence of partitions of $[0,t]$ such that $\mu
(\pi_{t}^{N})=\max\{|t_{i+1}^{N}-t_{i}^{N}|:i=0,\cdots,N-1\} \rightarrow0$,
the quadratic variation process of $B^{\mathbf{a}}$ is defined by%
\[
\langle B^{\mathbf{a}}\rangle_{t}=\lim_{\mu(\pi_{t}^{N})\rightarrow0}%
\sum_{j=0}^{N-1}(B_{t_{j+1}^{N}}^{\mathbf{a}}-B_{t_{j}^{N}}^{\mathbf{a}}%
)^{2}.
\]
For each fixed $\mathbf{a}$, $\mathbf{\bar{a}}\in\mathbb{R}^{d}$, the mutual
variation process of $B^{\mathbf{a}}$ and $B^{\mathbf{\bar{a}}}$ is defined by%
\[
\langle B^{\mathbf{a}},B^{\mathbf{\bar{a}}}\rangle_{t}=\frac{1}{4}[\langle
B^{\mathbf{a}+\mathbf{\bar{a}}}\rangle_{t}-\langle B^{\mathbf{a}%
-\mathbf{\bar{a}}}\rangle_{t}].
\]

\begin{definition}
\label{def2.6} Let $M_{G}^{0}(0,T)$ be the collection of processes in the
following form: for a given partition $\{t_{0},\cdot\cdot\cdot,t_{N}\}=\pi
_{T}$ of $[0,T]$,
\[
\eta_{t}(\omega)=\sum_{j=0}^{N-1}\xi_{j}(\omega)I_{[t_{j},t_{j+1})}(t),
\]
where $\xi_{i}\in L_{ip}(\Omega_{t_{i}})$, $i=0,1,2,\cdot\cdot\cdot,N-1$. For
$p\geq1$ and $\eta\in M_{G}^{0}(0,T)$, let $\Vert\eta\Vert_{H_{G}^{p}}=\{
\mathbb{\hat{E}}[(\int_{0}^{T}|\eta_{s}|^{2}ds)^{p/2}]\}^{1/p}$, $\Vert
\eta\Vert_{M_{G}^{p}}=\{ \mathbb{\hat{E}}[\int_{0}^{T}|\eta_{s}|^{p}%
ds]\}^{1/p}$ and denote by $H_{G}^{p}(0,T)$, $M_{G}^{p}(0,T)$ the completions
of $M_{G}^{0}(0,T)$ under the norms $\Vert\cdot\Vert_{H_{G}^{p}}$, $\Vert
\cdot\Vert_{M_{G}^{p}}$ respectively.
\end{definition}

\begin{theorem}
\label{the2.7} (\cite{DHP11,HP09}) There exists a weakly compact set
$\mathcal{P}\subset\mathcal{M}_{1}(\Omega_{T})$, the set of probability
measures on $(\Omega_{T},\mathcal{B}(\Omega_{T}))$, such that
\[
\mathbb{\hat{E}}[\xi]=\sup_{P\in\mathcal{P}}E_{P}[\xi]\ \ \text{for
\ all}\ \xi\in\mathcal{H}_{T}^{0}.
\]
$\mathcal{P}$ is called a set that represents $\mathbb{\hat{E}}$.
\end{theorem}

Let $\mathcal{P}$ be a weakly compact set that represents $\mathbb{\hat{E}}$.
For this $\mathcal{P}$, we define capacity%
\[
c(A):=\sup_{P\in\mathcal{P}}P(A),\ A\in\mathcal{B}(\Omega_{T}).
\]
A set $A\subset\Omega_{T}$ is polar if $c(A)=0$. A property holds
\textquotedblleft quasi-surely\textquotedblright\ (q.s. for short) if it holds
outside a polar set. In the following, we do not distinguish two random
variables $X$ and $Y$ if $X=Y$ q.s..

\begin{definition}
\label{def2.9} A process $\{M_{t}\}$ with values in $L_{G}^{1}(\Omega_{T})$ is
called a $G$-martingale if $\mathbb{\hat{E}}_{s}[M_{t}]=M_{s}$ for any $s\leq
t$.
\end{definition}

Let $S_{G}^{0}(0,T)=\{h(t,B_{t_{1}\wedge t},\cdot\cdot\cdot,B_{t_{n}\wedge
t}):t_{1},\ldots,t_{n}\in\lbrack0,T],h\in C_{b,Lip}(\mathbb{R}^{n+1})\}$. For
$p\geq1$ and $\eta\in S_{G}^{0}(0,T)$, set $\Vert\eta\Vert_{S_{G}^{p}}=\{
\mathbb{\hat{E}}[\sup_{t\in\lbrack0,T]}|\eta_{t}|^{p}]\}^{\frac{1}{p}}$.
Denote by $S_{G}^{p}(0,T)$ the completion of $S_{G}^{0}(0,T)$ under the norm
$\Vert\cdot\Vert_{S_{G}^{p}}$.

We consider the following type of $G$-BSDEs (in this paper we always use
Einstein convention):%
\begin{align}
Y_{t}  &  =\xi+\int_{t}^{T}f(s,Y_{s},Z_{s})ds+\int_{t}^{T}g_{ij}(s,Y_{s}%
,Z_{s})d\langle B^{i},B^{j}\rangle_{s}\nonumber\\
&  -\int_{t}^{T}Z_{s}dB_{s}-(K_{T}-K_{t}), \label{e3}%
\end{align}
where%

\[
f(t,\omega,y,z),g_{ij}(t,\omega,y,z):[0,T]\times\Omega_{T}\times
\mathbb{R}\times\mathbb{R}^{d}\rightarrow\mathbb{R}%
\]
satisfy the following properties:

\begin{description}
\item[(H1)] There exists some $\beta>1$ such that for any $y,z$,
$f(\cdot,\cdot,y,z),g_{ij}(\cdot,\cdot,y,z)\in M_{G}^{\beta}(0,T)$;

\item[(H2)] There exists some $L>0$ such that
\[
|f(t,y,z)-f(t,y^{\prime},z^{\prime})|+\sum_{i,j=1}^{d}|g_{ij}(t,y,z)-g_{ij}%
(t,y^{\prime},z^{\prime})|\leq L(|y-y^{\prime}|+|z-z^{\prime}|).
\]

\end{description}

For simplicity, we denote by $\mathfrak{S}_{G}^{\alpha}(0,T)$ the collection
of processes $(Y,Z,K)$ such that $Y\in S_{G}^{\alpha}(0,T)$, $Z\in
H_{G}^{\alpha}(0,T;\mathbb{R}^{d})$, $K$ is a decreasing $G$-martingale with
$K_{0}=0$ and $K_{T}\in L_{G}^{\alpha}(\Omega_{T})$.

\begin{definition}
\label{def3.1} Let $\xi\in L_{G}^{\beta}(\Omega_{T})$ and $f$ satisfy (H1) and
(H2) for some $\beta>1$. A triplet of processes $(Y,Z,K)$ is called a solution
of equation (\ref{e3}) if for some $1<\alpha\leq\beta$ the following
properties hold:

\begin{description}
\item[(a)] $(Y,Z,K)\in\mathfrak{S}_{G}^{\alpha}(0,T)$;

\item[(b)] $Y_{t}=\xi+\int_{t}^{T}f(s,Y_{s},Z_{s})ds+\int_{t}^{T}%
g_{ij}(s,Y_{s},Z_{s})d\langle B^{i},B^{j}\rangle_{s}-\int_{t}^{T}Z_{s}%
dB_{s}-(K_{T}-K_{t})$.
\end{description}
\end{definition}

\begin{theorem}
\label{the4.1} (\cite{HJPS}) Assume that $\xi\in L_{G}^{\beta}(\Omega_{T})$
and $f$, $g_{ij}$ satisfy (H1) and (H2) for some $\beta>1$. Then equation
(\ref{e3}) has a unique solution $(Y,Z,K)$. Moreover, for any $1<\alpha<\beta$
we have $Y\in S_{G}^{\alpha}(0,T)$, $Z\in H_{G}^{\alpha}(0,T;\mathbb{R}^{d})$
and $K_{T}\in L_{G}^{\alpha}(\Omega_{T})$.
\end{theorem}

\bigskip

\subsection{Comparison theorem of SDEs}

Let $\tau\in\lbrack0,T]$ and $\eta\in L_{G}^{2}(\Omega_{\tau})$, we consider
the following type SDE:%
\begin{equation}
X_{t}=\eta+\int_{\tau}^{t}b(s,X_{s})ds+\int_{\tau}^{t}h(s,X_{s})d\langle
B\rangle_{s}+\int_{\tau}^{t}\sigma(s,X_{s})dB_{s}+V_{t}-V_{\tau}, \label{4-1}%
\end{equation}
where $b$, $h$, $\sigma$ are given functions satisfying $b(\cdot,x)$,
$h(\cdot,x)$, $\sigma(\cdot,x)\in M_{G}^{2}(\tau,T)$ for each $x\in\mathbb{R}$
and the Lipschitz condition, i.e.,%
\[
|b(t,x)-b(t,x^{\prime})|+|h(t,x)-h(t,x^{\prime})|+|\sigma(t,x)-\sigma
(t,x^{\prime})|\leq K|x-x^{\prime}|;
\]
$(V_{t})_{t\in\lbrack\tau,T]}$ is a given RCLL process such that
$\mathbb{\hat{E}}[\sup_{t\in\lbrack\tau,T]}|V_{t}|^{2}]<\infty$. Peng
\cite{P10} proved that the above SDE has a unique solution $X\in M_{G}%
^{2}(\tau,T)$.

\begin{theorem}
Let $(X_{t}^{i})_{t\in\lbrack\tau,T]}$, $i=1,2,$ be the solutions of the
following SDEs:%
\[
X_{t}^{i}=\eta^{i}+\int_{\tau}^{t}b_{i}(s,X_{s}^{i})ds+\int_{\tau}^{t}%
h_{i}(s,X_{s}^{i})d\langle B\rangle_{s}+\int_{\tau}^{t}\sigma(s,X_{s}%
^{i})dB_{s}+V_{t}^{i}-V_{\tau}^{i}.
\]
If $\eta^{1}\geq\eta^{2}$, $b_{1}\geq b_{2}$, $h_{1}\geq h_{2}$, $V_{t}%
^{1}-V_{t}^{2}$ is an increasing process, then $X_{t}^{1}\geq X_{t}^{2}$.
\end{theorem}

\begin{proof}
We have%
\[
\hat{X}_{t}=\hat{\eta}+\int_{\tau}^{t}\hat{b}_{s}ds+\int_{\tau}^{t}\hat{h}%
_{s}d\langle B\rangle_{s}+\int_{\tau}^{t}\hat{\sigma}_{s}dB_{s}+\hat{V}%
_{t}-\hat{V}_{\tau},
\]
where $\hat{X}_{t}=X_{t}^{1}-X_{t}^{2}$, $\hat{\eta}=\eta^{1}-\eta^{2}$,
$\hat{b}_{s}=b_{1}(s,X_{s}^{1})-b_{2}(s,X_{s}^{2})$, $\hat{h}_{s}%
=h_{1}(s,X_{s}^{1})-h_{2}(s,X_{s}^{2})$, $\hat{\sigma}_{s}=\sigma(s,X_{s}%
^{1})-\sigma(s,X_{s}^{2})$, $\hat{V}_{t}=V_{t}^{1}-V_{t}^{2}$. For each given
$\varepsilon>0$, we can choose Lipschitz function $l(\cdot)$ such that
$I_{[-\varepsilon,\varepsilon]}\leq l(x)\leq I_{[-2\varepsilon,2\varepsilon]}%
$. Thus we have%
\[
b_{1}(s,X_{s}^{1})-b_{1}(s,X_{s}^{2})=(b_{1}(s,X_{s}^{1})-b_{1}(s,X_{s}%
^{2}))l(\hat{X}_{s})+b_{s}^{\varepsilon}\hat{X}_{s},
\]
where $b_{s}^{\varepsilon}=(1-l(\hat{X}_{s}))(b_{1}(s,X_{s}^{1})-b_{1}%
(s,X_{s}^{2}))\hat{X}_{s}^{-1}\in M_{G}^{2}(\tau,T)$ such that $|b_{s}%
^{\varepsilon}|\leq K$. It is easy to verify that%
\[
|(b_{1}(s,X_{s}^{1})-b_{1}(s,X_{s}^{2}))l(\hat{X}_{s})|\leq K|\hat{X}%
_{s}|l(\hat{X}_{s})\leq2K\varepsilon.
\]
Thus we can get%
\[
\hat{b}_{s}=b_{s}^{\varepsilon}\hat{X}_{s}+m_{s}+m_{s}^{\varepsilon},\ \hat
{h}_{s}=h_{s}^{\varepsilon}\hat{X}_{s}+n_{s}+n_{s}^{\varepsilon},\hat{\sigma
}_{s}=\sigma_{s}^{\varepsilon}\hat{X}_{s}+l_{s}^{\varepsilon},
\]
where $|m_{s}^{\varepsilon}|\leq2K\varepsilon$, $|n_{s}^{\varepsilon}%
|\leq2K\varepsilon$, $|l_{s}^{\varepsilon}|\leq2K\varepsilon$, $m_{s}%
=b_{1}(s,X_{s}^{2})-b_{2}(s,X_{s}^{2})\geq0$ and $n_{s}=h_{1}(s,X_{s}%
^{1})-h_{2}(s,X_{s}^{2})\geq0$. Let $\Gamma_{t}^{\varepsilon}$ be the solution
of the following SDE:%
\[
\Gamma_{t}^{\varepsilon}=1-\int_{\tau}^{t}b_{s}^{\varepsilon}\Gamma
_{s}^{\varepsilon}ds-\int_{\tau}^{t}[h_{s}^{\varepsilon}-(\sigma
_{s}^{\varepsilon})^{2}]\Gamma_{s}^{\varepsilon}d\langle B\rangle_{s}%
-\int_{\tau}^{t}\sigma_{s}^{\varepsilon}\Gamma_{s}^{\varepsilon}dB_{s}.
\]
By applying It\^{o}'s formula to $\Gamma_{t}^{\varepsilon}\hat{X}_{t}$, we
obtain that%
\[
\hat{X}_{t}\geq(\Gamma_{t}^{\varepsilon})^{-1}[\int_{\tau}^{t}m_{s}%
^{\varepsilon}\Gamma_{s}^{\varepsilon}ds+\int_{\tau}^{t}(n_{s}^{\varepsilon
}-\sigma_{s}^{\varepsilon}l_{s}^{\varepsilon})\Gamma_{s}^{\varepsilon}d\langle
B\rangle_{s}+\int_{\tau}^{t}l_{s}^{\varepsilon}\Gamma_{s}^{\varepsilon}%
dB_{s}].
\]
Note that $\Gamma_{t}^{\varepsilon}=\exp(-\int_{\tau}^{t}b_{s}^{\varepsilon
}ds-\int_{\tau}^{t}[h_{s}^{\varepsilon}-\frac{1}{2}(\sigma_{s}^{\varepsilon
})^{2}]d\langle B\rangle_{s}-\int_{\tau}^{t}\sigma_{s}^{\varepsilon}dB_{s})$,
thus we can get $\hat{X}_{t}\geq0$ by letting $\varepsilon\rightarrow0$.
\end{proof}

\subsection{Girsanov transformation}

We consider the following $G$-BSDE:%
\[
Y_{t}=\xi+\int_{t}^{T}b_{s}Z_{s}ds+\int_{t}^{T}d_{s}Z_{s}d\langle B\rangle
_{s}-\int_{t}^{T}Z_{s}dB_{s}-(K_{T}-K_{t}),
\]
where $(b_{t})_{t\leq T}$ and $(d_{t})_{t\leq T}$ are bounded processes. For
each $\xi\in L_{G}^{\beta}(\Omega_{T})$ with $\beta>1$, define%
\[
\mathbb{\tilde{E}}_{t}[\xi]=Y_{t}.
\]
By Theorem 5.1 in \cite{HJPS-1}, we know that $\mathbb{\tilde{E}}_{t}[\cdot]$
is a consistent sublinear expectation.

\begin{theorem}
(\cite{HJPS-1}) Let $(b_{t})_{t\leq T}$ and $(d_{t})_{t\leq T}$ be bounded
processes. Then

\begin{description}
\item[(1)] $\tilde{B}_{t}:=B_{t}-\int_{0}^{t}b_{s}ds-\int_{0}^{t}d_{s}d\langle
B\rangle_{s}$ is a $G$-Brownian motion under $\mathbb{\tilde{E}}$;

\item[(2)] for any decreasing $G$-martingale $\tilde{K}$ with $\tilde{K}%
_{0}=0$ and $\tilde{K}_{T}\in L_{G}^{\beta}(\Omega_{T})$ for some $\beta>1$,
we have $\mathbb{\tilde{E}}_{t}[\tilde{K}_{T}]=\tilde{K}_{t}$;

\item[(3)] the quadratic variation process of $\tilde{B}$ under
$\mathbb{\tilde{E}}$ equals to $\langle B\rangle$.
\end{description}
\end{theorem}

\begin{proof}
(1) and (2) can be found in \cite{HJPS-1}. We only prove (3). For each fixed
$t>0$, it is easy to check that%
\[
\lim_{n\rightarrow\infty}\mathbb{\hat{E}}[|\sum_{i=0}^{n-1}|\tilde{B}%
_{\frac{i+1}{n}t}-\tilde{B}_{\frac{i}{n}t}|^{2}-\langle B\rangle_{t}|^{2}]=0.
\]
By Proposition 3.7 in \cite{HJPS}, we can get $\mathbb{\tilde{E}}[|\sum
_{i=0}^{n-1}|\tilde{B}_{\frac{i+1}{n}t}-\tilde{B}_{\frac{i}{n}t}|^{2}-\langle
B\rangle_{t}|^{\beta}]\rightarrow0$ as $n\rightarrow\infty$ for some $\beta
\in(1,2)$. On the other hand, $\mathbb{\tilde{E}}[|\sum_{i=0}^{n-1}|\tilde
{B}_{\frac{i+1}{n}t}-\tilde{B}_{\frac{i}{n}t}|^{2}-\langle\tilde{B}\rangle
_{t}|^{\beta}]\rightarrow0$ as $n\rightarrow\infty$. Thus $\langle\tilde
{B}\rangle_{t}=\langle B\rangle_{t}$ under $\mathbb{\tilde{E}}$.
\end{proof}


\begin{thebibliography}{99}                                                                                               %


\bibitem {Avel1995}Avellaneda, M., Levy, A. and Paras A. (1995). Pricing and
hedging derivative securities in markets with uncertain volatilities. Appl.
Math. Finance 2, 73-88.

\bibitem {Bei}Beibner, P. (2012). Coherent price systems and
uncertainty-neutral valuation, working paper.

\bibitem {DenisMartini2006}Denis, L. and Martini, C. (2006) A Theoretical
Framework for the Pricing of Contingent Claims in the Presence of Model
Uncertainty, The Annals of Applied Probability, vol. 16, No. 2, pp 827-852.

\bibitem {DHP11}Denis, L., Hu, M. and Peng S.(2011) \emph{Function spaces and
capacity related to a sublinear expectation: application to $G$-Brownian
motion pathes,} Potential Anal., 34: 139-161.

\bibitem {EJ-1}L. Epstein and S. Ji, \emph{Ambiguous Volatility, Possibility
and Utility in Continuous Time,} (2011), arXiv:1103.1652.

\bibitem {EJ-2}L. Epstein and S. Ji, \emph{Ambiguous volatility and asset
pricing in continuous time}, Rev. Finan. Stud., (2013), forthcoming.

\bibitem {HJPS}Hu, M., Ji, S., Peng, S. and Song, Y. (2012) \emph{Backward
Stochastic Differential Equations Driven by $G$-Brownian Motion,}
arXiv:1206.5889v1 [math.PR].

\bibitem {HJPS-1}Hu, M., Ji, S., Peng, S. and Song, Y. (2012) \emph{Comparison
Theorem, Feynman-Kac Formula and Girsanov Transformation for BSDEs Driven by
G-Brownian Motion,} arXiv:1212.5403 [math.PR].

\bibitem {HP09}Hu, M. and Peng, S.(2009) \emph{On representation theorem of
G-expectations and paths of $G$-Brownian motion}. Acta Math. Appl. Sin. Engl.
Ser., 25,(3): 539-546, 2009.

\bibitem {MPP}A. Matoussi, L. Piozin and D. Possamai, \emph{Second-order BSDEs
with general reflection and Dynkin games under uncertainty}, (2012), arXiv:1212.0476.

\bibitem {MPZ}A. Matoussi, D. Possamai and C. Zhou, \emph{Robust Utility
Maximization in Non-dominated Models with 2BSDE}, the Uncertain Volatility
Model, (2012),\ To appear in Mathematical Finance.

\bibitem {NN}A. Neufeld and M. Nutz, \emph{Superreplication under Volatility
Uncertainty for Measurable Claims}, (2012),\ preprint.

\bibitem {Nu}M. Nutz, \emph{Random G-Expectations}, (2010),\ To appear in
Annals of Applied Probability.

\bibitem {Nutz}M. Nutz, \emph{A Quasi-Sure Approach to the Control of
Non-Markovian Stochastic Differential Equations}, Electronic Journal of
Probability, 17(2012), pp. 1-23.

\bibitem {NZ}M. Nutz and J. Zhang, \emph{Optimal Stopping under Adverse
Nonlinear Expectation and Related Games}, (2012), preprint.

\bibitem {Os}Osuka, E. (2011) \emph{Girsanov's formula for $G$-Brownian
motion,} arXiv:1106.2387v1 [math.PR].

\bibitem {Peng2004}Peng, S. (2004) \emph{Filtration consistent nonlinear
expectations and evaluations of contingent claims,} Acta Mathematicae
Applicatae Sinica, 20(2) 1--24.

\bibitem {Peng2005}Peng, S. (2005) \emph{Nonlinear expectations and nonlinear
Markov chains,} Chin. Ann. Math. 26B(2) 159--184.

\bibitem {P07a}Peng, S.(2007) \emph{$G$-expectation, $G$-Brownian Motion and
Related Stochastic Calculus of It\^o type}, Stochastic analysis and
applications, 541-567, Abel Symp., 2, Springer, Berlin.

\bibitem {P07b}Peng, S.(2007) \emph{$G$-Brownian Motion and Dynamic Risk
Measure under Volatility Uncertainty}, arXiv:0711.2834v1 [math.PR].

\bibitem {P08a}Peng, S.(2008) \emph{Multi-Dimensional $G$-Brownian Motion and
Related Stochastic Calculus under $G$-Expectation}, Stochastic Processes and
their Applications, 118(12): 2223-2253.

\bibitem {P08b}Peng, S.(2008) \emph{A New Central Limit Theorem under
Sublinear Expectations}, arXiv:0803.2656v1 [math.PR].

\bibitem {P10}Peng, S.(2010) \emph{Nonlinear Expectations and Stochastic
Calculus under Uncertainty}, arXiv:1002.4546v1 [math.PR].

\bibitem {PengICM2010}Peng, S.(2010) \emph{Backward Stochastic Differential
Equation, Nonlinear Expectation and Their Applications}, in Proceedings of the
International Congress of Mathematicians Hyderabad, India, 2010.

\bibitem {PSZ2012}Peng, S., Song, Y. and Zhang, J. (2012) \emph{A Complete
Representation Theorem for G-martingales, }Preprint,\emph{\ }arXiv:1201.2629v1.

\bibitem {STZ}Soner, M., Touzi, N. and Zhang, J.(2011) \emph{Martingale
Representation Theorem under G-expectation,} Stochastic Processes and their
Applications, 121: 265-287.

\bibitem {STZ11}Soner M, Touzi N, Zhang J.(2012) \emph{Wellposedness of Second
Order Backward SDEs,} Probability Theory and Related Fields, 153(1-2): 149-190.

\bibitem {Song11}Song, Y.(2011) \emph{Some properties on G-evaluation and its
applications to G-martingale decomposition,} Science China Mathematics, 54(2): 287-300.

\bibitem {Song12}Song, Y.(2012) \emph{Uniqueness of the representation for
$G$-martingales with finite variation,} Electron. J. Probab. 17 no. 24 1-15.

\bibitem {vorbrink}Vorbrink, J. 2010. Financial Markets with Volatility
Uncertainty. http://arxiv.org/abs/1012.1535.

\bibitem {XSZ}Xu, J., Shang, H, and Zhang, B. (2011) \emph{A Girsanov type
theorem under G-framework,} Stoch. Anal. Appl., 29: 386--406.
\end{thebibliography}
\end{document}